\UseRawInputEncoding
\documentclass{article}
\usepackage{amsfonts}
\usepackage{amscd}
\usepackage{latexsym,bm,amssymb}
\usepackage{amsmath}
\usepackage{amsthm}
\usepackage{indentfirst}
\usepackage[sort&compress,numbers]{natbib}
\usepackage{color}
\usepackage{multirow}
\usepackage{graphicx}
\usepackage{epstopdf}
\newtheorem{Def}{Definition}[section]
\newtheorem{thy}{Theorem}[section]
\newtheorem{lem}{Lemma}[section]
\newtheorem{rema}{Remark}[section]
\newtheorem{exm}{Example}[section]

\newtheorem{pro}{Proposition}[section]

\newtheorem{algo}{Algorithm}[section]

\usepackage[perpage,symbol*]{footmisc}
\usepackage{footmisc}

\usepackage[pdftex,unicode,colorlinks,linkcolor=blue]{hyperref}

\pdfoutput=1
\begin{document}
\newcommand{\re}[1]{\begin{rema} {\rm{#1}} \end{rema}}
\newcommand{\de}[1]{\begin{Def} {\rm{#1}} \end{Def}}
\newcommand{\ex}[1]{\begin{exm} {\rm{#1}} \end{exm}}
\newcommand{\norm}[1]{\parallel #1\parallel}
\newcommand{\seq}[1]{\langle #1\rangle}
\newcommand{\sty}{\displaystyle}
\newcommand{\ra}{\rightarrow}
\newcommand{\Ra}{\Rightarrow}
\newcommand{\n}{ {\mathcal N} }
\def\QED{\hfill $\Box$\smallskip}

\title{{\Large {\bf An inertial Tseng's  extragradient method for solving multi-valued variational inequalities with one projection}}
\thanks{This work is partially supported by the National Natural Science Foundation of
China (No. 11771350), Basic and Advanced Research Project of CQ CSTC (Nos. cstc2020jcyj-msxmX0738 and cstc2018jcyjAX0605) }
\author{Changjie Fang,\thanks{Corresponding author, E-mail address: fangcj@cqupt.edu.cn} \, Ruirui Zhang,\,Shenglan Chen\\
\small\it Key Lab of Intelligent Analysis and Decision on Complex Systems, \\\small\it Chongqing University
of Posts and Telecommunications, Chongqing 400065, China\\\small\it  School of Science, Chongqing University
of Posts and Telecommunications,\\
\small\it Chongqing
400065,
 China}}
\date{}
\maketitle
\vspace*{-9mm}
\begin{center}
\begin{minipage}{5.5in}
{\bf Abstract.}\quad In this paper, we introduce an inertial  Tseng's  extragradient method for solving multi-valued variational inequalits, in which only one projection is needed at each iterate. We also obtain the strong convergence results of the proposed algorithm, provided that the multi-valued mapping is continuous and pseudomonotone with nonempty compact convex values. Moreover, numerical simulation results illustrate the efficiency of our method when compared to existing methods.
\\ \ \\
\end{minipage}
\end{center}
\section{Introduction}

In this paper, we consider the following multi-valued variational inequality, denoted by $MVI(A,C)$: to find $x^\ast\in C$ and $w^{\ast} \in A(x^{\ast})$ such that
\begin{align}\label{z01}
 \langle w^\ast, y-x^{\ast}\rangle \geq 0 \qquad \forall\,\, y \in C,
\end{align}
where $C$ is a non-empty closed convex set in $\mathbb{R}^{n}$, $A$ is a multi-valued mapping form $\mathbb{R}^{n}$ into $2^{\mathbb{R}^n}$ with nonempty values, $\langle\cdot , \cdot\rangle$  and $\|\cdot\|$ denote the usual inner product and norm in $\mathbb{R}^{n}$, respectively. If $A$ is a single-valued mapping, then $MVI(A,C)$ reduces to the classic variational inequality problem.

Variational inequality theory has emerged as an important tool in studying obstacle problems,
unilateral problems in mechanics, equilibrium problems, and so on; see \cite{FC2014,FH11,FH13,He06,SS99,TV2018,WXZ03,ZM17}
and the references therein. In order to explore relevant convergent results and analyze error estimates,
many methods for solving variational inequality problem \eqref{z01} have been proposed, in which the most popular method is the projection-type one. An important projection algorithm for solving variational inequalities is
the Extragradient Method proposed by Korpelevich \cite{Kor76}; In \cite{Kor76}, there is the need to calculate two projections
onto $C$, and convergence is proved under the assumption of Lipschitz continuity and monotonicity. It is well known that, if $C$ is a general closed convex set, this might
be computationally expensive and hence it will affect the efficiency of the proposed algorithms. To overcome the difficulty, Censor et al. \cite{CGR11} proposes a subgradient extragradient algorithm for solving single-valued variational inequality, in which  the second projection is onto $C$ instead of  the half-space; see also \cite{CGR12,KS14}.  We note that the above algorithms needs at least  two projections per iteration. Further,  One-projection methods for solving single-valued variational inequality problems are proposed; see for example \cite{Mai09,Mal2015,Ts00}.

Projection-type methods for solving multi-valued variational inequality have been proposed. Li and He \cite{LH2009} proposed a projection algorithm for solving multi-valued variational inequality in which the hyperplane strictly separates the current iterate from the
solution set; see also \cite{FH11}. Xia and Huang \cite{XH11} studied a projection-proximal point algorithm for solving  multi-valued variational inequalities  in Hilbert spaces and obtained the weak convergence result under the assumption of pseudomonotonicity. Further, Fang and Chen \cite{FC2014} extended the  subgradient
extragradient algorithm in \cite{CGR11} to solve multi-valued variational inequality\eqref{z01}. Recently, Burachik and Milln \cite{BM19} suggesteded a projection-type algorithm for solving \eqref{z01}, in which the next iterate is a projection of the  initial point onto the intersection of some suitable convex subsets. He et al. \cite{HHL19} proposed  two  projection-type algorithms for solving the multivalued variational inequality and studied the convergence of
the proposed algorithms.  Inspired by Fang et al.\cite{FC2014,He97}, Dong et al.\cite{DLY17}  presented a projection and
contraction method for solving multi-valued variational inequality \eqref{z01} and prove the strong convergence of the proposed algorithm.

The inertial-type methods originate from an implicit discretization method of the heavy-ball with
friction(HBF) system, the main feature of which is that each new iterate point depends on the previous two iterates(\cite{Al00}). Subsequently, this inertial technique was extended to solve the inclusion problem of maximal monotone operators (\cite{AA01}). Since then, there has been increasing interest in studying
inertial-type algorithms; see, for example, inertial forward-backward splitting methods(\cite{APR14,OCBP14}), inertial
Douglas-Rachford splitting method (\cite{BCH15}), inertial ADMM (\cite{CCMY15}), inertial-type methods for variational
inequalities (\cite{CMY15,ZFC18}).

Motivated by the recent work mentioned above, in this paper, we present an  inertial Tseng's  extragradient method for solving multi-valued variational inequalities, in which only one projection is needed at each iterate; see Step 3 in Algorithm \ref{alg}. In our method, the projection onto the hypeplane in \cite{FC2014} is replaced by the Tseng's term; see Step 4 in Algorithm \ref{alg}. In addition, the mapping $A$ is assumed to be  pseudomonotone with nonempty
compact convex values. Under those assumptions above, we prove that the iterative sequence generated by our method converges strongly to a solution of the multi-valued variational inequality \eqref{z01}.  We also present numerical results of the proposed method.

This paper is organized as follows. In Section \ref{pre}, we present  definitions and auxiliary material. In Section \ref{res}, we describe our algorithm and investigate the global convergence of our method. Numerical experiments are reported in Section \ref{nume}.
\section{Preliminaries}\label{pre}
\setcounter{equation}{0}
In this section, we introduce some basic concepts which are used in this paper.

The multi-valued mapping $A: \mathbb{R}^n\rightarrow 2^{\mathbb{R}^n}$  is said to be upper semicontinuous at $x\in C$ if for every open set $V$ containing $A(x)$, there is an open set $U$ containing $x$ such that $A(y)\in V$ for all $y\in C\bigcap U$. $A$ is said to be lower semicontinuous at $x\in C$ if given any sequence $x_{n}$ converging to $x$ and any $y\in A(x)$, there exists a sequence $y_{n}\in A(x_{n})$ that converges to $y$. $A$ is said to be continuous at $x\in C$ if it is both upper semicontinuous and lower semicontinuous at $x$.

Let the set $C$ be given by
\begin{align*}
C:=\{x\in \mathbb{R}^{n}\|g(x)\leq0\},
\end{align*}
where $g:\mathbb{R}^{n}\rightarrow \mathbb{R}$ is a convex function. We denote the subdifferential of $g$ at a point $x$ by
\begin{align*}
\partial g(x):=\{w\in \mathbb{R}^{n}|g(y)\geq g(x)+\langle w,y-x\rangle,\forall y \in \mathbb{R}^{n}\}.
\end{align*}

The multi-valued mapping $A$ is called monotone on $C$, if for any $x,y\in C$,
\begin{align*}
\langle u-\nu,x-y\rangle\geq0,\,\,\,\ \forall u\in A(x),\,\,\,\ \forall \nu\in A(y).
\end{align*}

The multi-valued mapping $A$ is called pseudomonotone on $C$, if for any $x,y\in C$,
\begin{equation}
\begin{aligned}
\langle\nu,x-y\rangle\geq0,\,\,\,\ \exists\nu\in A(y)\ \Longrightarrow\langle u,x-y\rangle\geq0,\,\,\,\ \forall u\in A(x).
\label{Z02}
\end{aligned}
\end{equation}

 Denote by $S$  the solution set of the multi-value variation inequality \eqref{z01}. Throughout this paper, we assume that the solution set $S$ is nonempty satisfying the following property:
\begin{equation}
\begin{aligned}
\langle w,y-x\rangle\geq0,\,\,\,\ \forall y\in C\,\,\,\ \forall w\in A(y)\,\,\,\ \forall x\in S.
\label{Z03}
\end{aligned}
\end{equation}
The property \eqref{Z03} holds if $A$ is pseudomonotone on $C$.

The projection of a point $x\in \mathbb{R}^{n}$ onto a closed set $C$ is defined as
\begin{align*}
P_{C}(x)=\textrm{argmin}_{y\in C}\parallel y-x\parallel.
\end{align*}

\begin{lem}\label{Lem2.1}(\textup{\cite{Zara71}}) Let $C$ be a closed convex subset of $\mathbb{R}^{n}$. For any $x,y\in \mathbb{R}^{n}$  and $z\in C$, the following statements hold,
\begin{itemize}
\item[(i)]$\langle x-P_{C}(x),z-P_{C}(x)\rangle\leq0$;
\item[(ii)]$\|P_{C}(x)-P_{C}(y)\|^{2}\leq\|x-y\|^{2}-\|P_{C}(x)-x+y-P_{C}(y)\|^{2}$.
\end{itemize}
\end{lem}

\begin{pro}\textup{\cite{LH2009}}\label{pro01} $x\in C$, and $w\in A(x)$ solve the problem \eqref{z01} if and only if
\begin{align*}
r_{\mu}(x,w):=x-P_{C}(x-\mu w)=0.
\end{align*}
\end{pro}
\begin{pro}(\textup{\cite{FC2014}})\label{pro02} For any $x\in \mathbb{R}^{n}$, $w\in A(x)$ and $\mu>0$,
\begin{align*}
\min\{1,\mu\}\|r_{1}(x,w)\|\leq\parallel r_{\mu}(x,w)\parallel\leq \max\{1,\mu\}\|r_{1}(x,w)\|.
\end{align*}
\end{pro}
\begin{lem}\label{Lem2.2}(\textup{\cite{ZFC18}})For all $x,y\in \mathbb{H}$ and $\lambda \in[0,1]$,
\begin{align*}
\parallel\lambda x+(1-\lambda)y\parallel^{2}=\lambda\parallel x\parallel^{2}+(1-\lambda)\|y\|^2-\lambda(1-\lambda)\|x-y\|^{2},
\end{align*}
where $\mathbb{H}$ is a real Hilbert space.
\end{lem}
\begin{lem}\label{Lem2.3}(\textup{\cite{AA01}}) Let $\{\varphi_{n}\}$, $\{\theta_{n}\}$, and $\{\alpha_{n}\}$ be sequences in $[0,+\infty)$, such that
\begin{align*}
\varphi_{n+1}\leq\varphi_{n}+\alpha_{n}(\varphi_{n}-\varphi_{n-1})+\theta_{n}\,\,\,\ \forall n\geq1, \,\,\,\ \sum_{n=1}^{+\infty}\theta_{n}<+\infty
\end{align*}
and there exists a real number $\alpha$ with $0\leq\alpha_{n}\leq\alpha<1$ for all $n\in N$. Then, the following hold:
\begin{itemize}
\item[(i)]$\Sigma_{n=1}^{+\infty}[\varphi_{n}-\varphi_{n-1}]_{+}<+\infty, where [t]_{+}=\max\{t,0\}$.
\item[(ii)]There exists $\varphi^{*}\in [0,+\infty)$, such that $\lim_{n\rightarrow+\infty}\varphi_{n}=\varphi^{*}$.\\
\end{itemize}
\end{lem}
\begin{lem}\label{Lem2.4} (\textup{\cite{Op67}}). Let $C$ be a nonempty set of $\mathbb{H}$ and ${x_{n}}$ be a sequence in $\mathbb{H}$ such that the following two conditions hold:
\begin{itemize}
\item[(i)] For every $x\in C$, $\lim_{n\rightarrow\infty}\|x_{n}-x\|$ exists.
\item[(ii)] Every sequential weak cluster point of $x_{n}$ is in $C$. then, $x_{n}$ converges weakly to a point in $C$. \\
\end{itemize}
\end{lem}
\section{Main results}\label{res}
\indent
\par
\setcounter{equation}{0}In this section, we introduce the inertial Tseng's extragradient algorithm for solving the multivalued variational inequality problems.
In order to find a point of the set $C$, we have the following procedure.

\textbf{Procedure A \cite{Kon98}}

Data  A Point $x\in \mathbb{R}^{n}$.

Output A point $R(x)$.

step 0. If $x\in C$, set $R(x)=x$. Otherwise, set $y_{0}=x$, $n=0$.

Step 1. Choose $w_{n}\in \partial g(y_{n})$, set $y_{n+1}-2g(y_{n}) \frac{w_{n}}{\|w_{n}\|^{2}}$.

Step 2. If $y_{n+1}\in C$, set $R(x)=y_{n+1}$ and stop. Otherwise, set $n=n+1$ go to Step 1. We get the

following results from Procedure A.

\begin{pro}\cite{KF84}\label{pro03} The number of iterations in Procedure A is finite.
\end{pro}
\begin{pro}\cite{Kon98}\label{pro04} Let $x\in \mathbb{R}^{n}$, we have
\begin{align*}
\|R(x)-y\|\leq\|x-y\|, \,\,\,\ \forall y\in C, \,\,\,\ R(x)\in C.
\end{align*}
\end{pro}
\begin{algo}\label{alg}
Choose $\tilde{x}_{0}\in \mathbb{R}^{n}$, $\tilde{x}_{1}\in \mathbb{R}^{n}$ and two parameters $\mu,\gamma\in(0,1)$. Set $n=1$

Step 1. Apply Procedure A with $x=\tilde{x}_{0}$ and set $x_{0}=R(\tilde{x}_{0})$.

Step 2. Apply Procedure A with $x=\tilde{x}_{n}$ and set $x_{n}=R(\tilde{x}_{n})$.

Step 3. Let $w_{n}=x_{n}+\alpha_{n}(x_{n}-x_{n-1})$, choose $u_{n}\in A(w_{n})$, and compute
\begin{align*}
y_{n}=P_{C}(w_{n}-\lambda_{n}u_{n}),
\end{align*}
where $\lambda_{n}=\gamma^{m_{n}}$ and $m_{n}$ is the smallest nonnegative integer $m$ such that
\begin{align}
\nu_{n}\in A(P_{C}(w_{n}-\gamma^{m}u_{n})).\label{ZZ02}
\end{align}
\begin{align}
\gamma^{m}\|u_{n}-\nu_{n}\|\leq\mu\|r_{\gamma^{m}}(w_{n},u_{n})\|.\label{ZZ03}
\end{align}
If $r_{\lambda_{n}}(w_{n},u_{n})=0$, then stop.

Step 4. Compute
\begin{align}
\tilde{x}_{n+1}=y_{n}-\lambda_{n}(\nu_{n}-u_{n}).\label{Z03a}
\end{align}
Set $n: = n + 1$ and return to Step 2.
\end{algo}
We first show that Algorithm \ref{alg} is well defined.\\

\begin{lem} Suppose that the assumption \eqref{Z03} holds, then for any $\gamma\in(0,1)$ and $x_{n}\in C$, the linesearch procedure in Algrithm \ref{alg} is well defined.
\end{lem}
\begin{proof} If $r_{1}(w_{n},u_{n})=0$, then by Proposition \ref{pro02} we have $r_{\gamma^{m}}(w_{n},u_{n})=0$, i.e., $w_n=P_C(w_n-\gamma^mu_n)$ and hence we can take $\nu_n=u_n$ which satisfies \eqref{ZZ02} and \eqref{ZZ03}.\\
Assume now that $\|r_{1}(w_{n},u_{n})\|>0$. Suppose that for all $m$ and $\nu\in A(y_{m})=A(P_{C}(w_{n}-\gamma^{m}u_{n}))$, we have
\begin{align}
&\gamma^{m}\|u_{n}-\nu\|>\mu\|r_{\gamma^{m}}(w_{n},u_{n})\|,\label{ZZ04}
\end{align}\\
i.e.,
\begin{align}
\|u_{n}-\nu\|>\frac{\mu}{\gamma^{m}}\|r_{\gamma^{m}}(w_{n},u_{n})\|\geq\frac{\mu}{\gamma^{m}}\min\{1,\gamma^{m}\}\|r_{1}(w_{n},u_{n})\|=\mu\|r_{1}(w_{n},u_{n})\|,\label{ZZ04a}
\end{align}
where the second inequality follows from Proposition \ref{pro02} and the equality follows from $\gamma\in(0,1)$ and $m\geq0$.

We now consider the two cases, $w_{n}\in C$ and $w_{n}\notin C$.

(i) If $w_{n}\in C$. Since $P_{C}(\cdot)$ is continuous, $y_{m}=P_{C}(w_{n}-\gamma^{m}u_{n})\rightarrow w_{n}(m\rightarrow \infty)$. Since $A$ is lower semicontinuous, $u_{n}\in A(w_{n})$ and $y_{m}\rightarrow w_{n}(m\rightarrow\infty)$, there is $\nu_{m}\in A(y_{m})$ such that $\nu_{m}\rightarrow u_{n}(m\rightarrow \infty)$. Therefore, from \eqref{ZZ04} we have
\begin{align}
\|u_{n}-\nu_{m}\|>\mu\|r_{1}(w_{n},u_{n})\|,\,\,\,\,\ \forall m.\label{ZZ05}
\end{align}
Letting $m\rightarrow \infty$ in \eqref{ZZ05}, we have
\begin{align*}
0=\|u_{n}-u_{n}\|\geq\mu\|r_{1}(w_{n},u_{n})\|>0.
\end{align*}
This is a contradiction.

(ii) If $w_{n}\notin C$, then $\|r_{\gamma^{m}}(w_{n},u_{n})\|\rightarrow\|w_{n}-P_{C}(w_{n})\|\neq0(m\rightarrow\infty)$. Letting $m\rightarrow\infty$ in \eqref{ZZ04}, we have
\begin{align*}
0=\gamma^{m}\|u_{n}-\nu\|\geq\mu\|w_{n}-P_{C}(w_{n})\|>0,
\end{align*}
being $A$ continuous. This is a contradiction. Thus, Algorithm \ref{alg} is well defined and implementable.\\
\end{proof}

Next we show that the stopping criterion in Step 3 is valid.
\begin{lem}\label{Lem3.1} If $r_{\lambda_{n}}(w_{n},u_{n})=0$ in Algorithm 3.1, then $w_{n}\in S$.
\end{lem}
\begin{proof} If $r_{\lambda_{n}}(w_{n},u_{n})=0$, then $w_{n}=P_{C}(w_{n}-\lambda_{n}u_{n})$. Since $\lambda_{n}>0$, it follows from Proposition \ref{pro01} that $w_{n}\in S$.
\end{proof}


The following two lemmas play an important role in proving the convergence of Algorithm \ref{alg}.
\begin{lem}\label{Lem4.1}
Let $\{x_{n}\}$ be a sequence generated by Algorithm 3.1. Then for every $x^*\in S$��
\begin{align}\label{Z04}
\parallel x_{n+1}-x^*\parallel^{2}\leq\parallel w_{n}-x^*\parallel^{2}-(1-\mu^{2})\parallel r_{\lambda_{n}}(w_{n},u_{n})\parallel^{2}
\end{align}
\end{lem}
\begin{proof}From \eqref{Z03a} we have
\begin{equation}\label{Z05}
\begin{split}
\|\tilde{x}_{n+1}-x^*\|^{2}&=\|y_{n}-\lambda_{n}(\nu_{n}-u_{n})-x^*\|^{2}\\
&=\|y_{n}-x^*\|^2+\lambda_{n}^2\|\nu_{n}-u_{n}\|^2-2\lambda_{n}\langle y_{n}-x^*,\nu_{n}-u_{n}\rangle\\
&=\|w_{n}-x^*\|^2+\|w_{n}-y_{n}\|^2+2\langle y_{n}-w_{n},w_{n}-x^*\rangle+\lambda_{n}^2\|\nu_{n}-u_{n}\|^2-2\lambda_{n}\langle y_{n}-x^*,\nu_{n}-u_{n}\rangle\\
&=\|w_{n}-x^*\|^2+\|w_{n}-y_{n}\|^2-2\langle y_{n}-w_{n},y_{n}-w_{n}\rangle+2\langle y_{n}-w_{n},y_{n}-x^*\rangle\\
&+\lambda_{n}^2\|\nu_{n}-u_{n}\|^2-2\lambda_{n}\langle y_{n}-x^*,\nu_{n}-u_{n}\rangle\\
&=\|w_{n}-x^*\|^2-\|w_{n}-y_{n}\|^2+2\langle y_{n}-w_{n},y_{n}-p\rangle
+\lambda_{n}^2\|\nu_{n}-u_{n}\|^2-2\lambda_{n}\langle y_{n}-x^*,\nu_{n}-u_{n}\rangle.\\
\end{split}
\end{equation}
Since $y_{n}=P_{C}(w_{n}-\lambda_{n}u_{n})$,
\begin{align*}
\langle y_{n}-w_{n}+\lambda_{n}u_{n},y_{n}-x^*\rangle\leq0,
\end{align*}
or equivalently
\begin{equation}\label{Z06}
\begin{split}
\langle y_{n}-w_{n},y_{n}-x^*\rangle\leq-\lambda_{n}\langle u_{n},y_{n}-x^*\rangle.\\
\end{split}
\end{equation}
Form \eqref{Z05} and \eqref{Z06}, we get
 \begin{equation}\label{Z07}
\begin{split}
\|\tilde{x}_{n+1}-x^*\|^{2}&\leq\|w_{n}-x^*\|^{2}-\|w_{n}-y_{n}\|^{2}-2\lambda_{n}\langle u_{n},y_{n}-x^*\rangle+\lambda_{n}^{2}\|\nu_{n}-u_{n}\|^{2}-2\lambda_{n}\langle y_{n}-x^*,\nu_{n}-u_{n}\rangle\\
&=\|w_{n}-x^*\|^{2}-\|w_{n}-y_{n}\|^{2}+\lambda_{n}^{2}\|\nu_{n}-u_{n}\|^{2}-2\lambda_{n}\langle y_{n}-x^*,\nu_{n}\rangle\\
&\leq\|w_{n}-x^*\|^{2}-\|r_{\lambda_{n}}(w_{n},u_{n})\|^{2}+\mu^{2}\|r_{\lambda_{n}}(w_{n},u_{n})\|^{2}-2\lambda_{n}\langle y_{n}-x^*,\nu_{n}\rangle\\
&\leq\|w_{n}-x^*\|^{2}-(1-\mu^{2})\|r_{\lambda_{n}}(w_{n},u_{n})\|^{2}-2\lambda_{n}\langle y_{n}-x^*,\nu_{n}\rangle.\\
\end{split}
\end{equation}
Since $\nu_{n}\in A(y_{n})$ and $ x^*\in S$, it follows from \eqref{Z03} that
 \begin{equation}\label{Z08}
\begin{split}
\langle \nu_{n},y_{n}-x^*\rangle\geq0.\\
\end{split}
\end{equation}
Combining \eqref{Z07} and \eqref{Z08}, we have
\begin{align*}
\|\tilde{x}_{n+1}-x^*\|^{2}\leq\|w_{n}-x^*\|^{2}-(1-\mu^{2})\|r_{\lambda_{n}}(w_{n},u_{n})\|^{2},\\
\end{align*}
and hence from Proposition \ref{pro04} we get
\begin{align*}
\|x_{n+1}-x^*\|^{2}\leq\|\tilde{x}_{n+1}-x^*\|^{2}\leq\|w_{n}-x^*\|^{2}-(1-\mu^{2})\|r_{\lambda_{n}}(w_{n},u_{n})\|^{2}.
\end{align*}
This completes the proof.
\end{proof}

\begin{lem}\label{Th4.1} Assume that the sequence $\{\alpha_{n}\}$ is non-decreasing satisfying $0\leq\alpha_{n}\leq\alpha$ and
 \begin{equation}\label{Z09}
\begin{split}
\alpha<1-\frac{4}{\sqrt{8\tau+1}+3},\\
\end{split}
\end{equation}
where $\tau=\frac{2}{\mu+1}-1$, and that $x^*\in S$. Then,
\begin{itemize}
\item[(1)]$\lim_{n\rightarrow\infty}\|x_{n}-x^*\|$ exists;
\item[(2)]$\lim_{n\rightarrow\infty}\|w_{n}-x_{n}\|=0$.
\end{itemize}
\end{lem}
\begin{proof} By the definition of $\tilde{x}_{n+1}$, we have
\begin{align*}
\|\tilde{x}_{n+1}-y_{n}\|&=\|y_{n}-\lambda_{n}(\nu_{n}-u_{n})-y_{n}\|\\
&\leq\lambda_{n}\|\nu_{n}-u_{n}\|\\
&\leq\mu\|r_{\lambda_{n}}(w_{n},u_{n})\|.\\
\end{align*}
Therefore,
\begin{align*}
\|\tilde{x}_{n+1}-w_{n}\|&\leq\|\tilde{x}_{n+1}-y_{n}\|+\|y_{n}-w_{n}\|\\
&\leq(1+\mu)\|r_{\lambda_{n}}(w_{n},u_{n})\|,\\
\end{align*}
which implies
\begin{equation}\label{Z10}
\begin{split}
\|r_{\lambda_{n}}(w_{n},u_{n})\|\geq\frac{1}{1+\mu}\|\tilde{x}_{n+1}-w_{n}\|.\\
\end{split}
\end{equation}
Let $x^*\in S$ . By Lemma \ref{Lem4.1}, we have
\begin{equation}\label{Z11}
\begin{split}
\|x_{n+1}-x^*\|^{2}\leq\|w_{n}-x^*\|^{2}-(1-\mu^{2})\|r_{\lambda_{n}}(w_{n},u_{n})\|^{2}.
\end{split}
\end{equation}
From \eqref{Z10} and \eqref{Z11}, we have
\begin{equation}\label{Z12}
\begin{split}
\|x_{n+1}-x^*\|^{2}&\leq\|w_{n}-x^*\|^{2}-\frac{1-\mu^{2}}{(1+\mu)^{2}}\|\tilde{x}_{n+1}-w_{n}\|^{2}\\
&=\|w_{n}-x^*\|^{2}-(\frac{2}{\mu+1}-1)\|\tilde{x}_{n+1}-w_{n}\|^{2}\\
&=\|w_{n}-x^*\|^{2}-\tau\|\tilde{x}_{n+1}-w_{n}\|^{2}\\
&\leq\|w_{n}-x^*\|^{2}-\tau\|{x}_{n+1}-w_{n}\|^{2}.\\
\end{split}
\end{equation}
By the definition of $w_{n}$, we have
\begin{equation}\label{Z13}
\begin{split}
\|w_{n}-x^*\|^{2}&=\|x_{n}+\alpha_{n}(x_{n}-x_{n-1})-x^*\|^{2}\\
&=\|(1+\alpha_{n})(x_{n}-x^*)-\alpha_{n}(x_{n-1}-x^*)\|^{2}\\
&=(1+\alpha_{n})\|x_{n}-x^*\|^{2}-\alpha_{n}\|x_{n-1}-x^*\|^{2}+\alpha_{n}(1+\alpha_{n})\|x_{n}-x_{n-1}\|^{2}.\\
\end{split}
\end{equation}
Thus, it follows from \eqref{Z12} and \eqref{Z13} that
\begin{equation}\label{Z14}
\begin{split}
\|x_{n+1}-x^*\|^{2}&\leq(1+\alpha_{n})\|x_{n}-x^*\|^{2}-\alpha_{n}\|x_{n-1}-x^*\|^{2}+\alpha_{n}(1+\alpha_{n})\|x_{n}-x_{n-1}\|^{2}\\
&\leq(1+\alpha_{n})\|x_{n}-x^*\|^{2}-\alpha_{n}\|x_{n-1}-x^*\|^{2}+2\alpha\|x_{n}-x_{n-1}\|^{2}.\\
\end{split}
\end{equation}
Also,
\begin{equation}\label{Z15}
\begin{split}
\|x_{n+1}-w_{n}\|^{2}&=\|x_{n+1}-x_{n}-\alpha_{n}(x_{n}-x_{n-1})\|^{2}\\
&=\|x_{n+1}-x_{n}\|^{2}+\alpha_{n}^{2}\|x_{n}-x_{n-1}\|^{2}-2\alpha_{n}\langle x_{n+1}-x_{n},x_{n}-x_{n-1}\rangle\\
&\geq\|x_{n+1}-x_{n}\|^{2}+\alpha_{n}^{2}\|x_{n}-x_{n-1}\|^{2}-2\alpha_{n}\|x_{n+1}-x_{n}\|\|x_{n}-x_{n-1}\|\\
&\geq(1-\alpha_{n})\|x_{n+1}-x_{n}\|^{2}+(\alpha_{n}^{2}-\alpha_{n})\|x_{n}-x_{n-1}\|^{2}.\\
\end{split}
\end{equation}
Combining \eqref{Z12}, \eqref{Z13} and \eqref{Z15}, we have
\begin{equation}\label{Z16}
\begin{split}
\|x_{n+1}-x^*\|^{2}
&\leq(1+\alpha_{n})\|x_{n}-x^*\|^{2}-\alpha_{n}\|x_{n-1}-x^*\|^{2}+\alpha_{n}(1+\alpha_{n})\|x_{n}-x_{n-1}\|^{2}\\
&-\tau(1-\alpha_{n})\|x_{n+1}-x_{n}\|^{2}-\tau(\alpha_{n}^{2}-\alpha_{n})\|x_{n}-x_{n-1}\|^{2}\\
&=(1+\alpha_{n})\|x_{n}-x^*\|^{2}-\alpha_{n}\|x_{n-1}-x^*\|^{2}-\tau(1-\alpha_{n})\|x_{n+1}-x_{n}\|^{2}\\
&+[\alpha_{n}(1+\alpha_{n})-\tau(\alpha_{n}^{2}-\alpha_{n})]\|x_{n}-x_{n-1}\|^{2}\\
&=(1+\alpha_{n})\|x_{n}-x^*\|^{2}-\alpha_{n}\|x_{n-1}-x^*\|^{2}-\sigma_{n}\|x_{n+1}-x_{n}\|^{2}+\delta_{n}\|x_{n}-x_{n-1}\|^{2},\\
\end{split}
\end{equation}
where $\sigma_{n}=\tau(1-\alpha_{n})>0$ and $\delta_{n}=\alpha_{n}(1+\alpha_{n})-\tau(\alpha_{n}^{2}-\alpha_{n})\geq0$.\\
Set
\begin{align*}
\Phi_{n}&=\|x_{n}-x^*\|^{2}-\alpha_{n}\|x_{n-1}-x^*\|^{2}+\delta_{n}\|x_{n}-x_{n-1}\|^{2},\\
\end{align*}
and hence
\begin{align*}
\Phi_{n+1}&=\|x_{n+1}-x^*\|^{2}-\alpha_{n+1}\|x_{n}-x^*\|^{2}+\delta_{n+1}\|x_{n+1}-x_{n}\|^{2}.\\
\end{align*}
Therefore, from \eqref{Z16} we have
\begin{equation}\label{Z17}
\begin{split}
\Phi_{n+1}-\Phi_{n}&=\|x_{n+1}-x^*\|^{2}-(1+\alpha_{n+1})\|x_{n}-x^*\|^{2}+\alpha_{n}\|x_{n-1}-x^*\|^{2}\\
&+\delta_{n+1}\|x_{n+1}-x_{n}\|^{2}-\delta_{n}\|x_{n}-x_{n-1}\|^{2}\\
&\leq\|x_{n+1}-x^*\|^{2}-(1+\alpha_{n})\|x_{n}-x^*\|^{2}+\alpha_{n}\|x_{n-1}-x^*\|^{2}\\
&+\delta_{n+1}\|x_{n+1}-x_{n}\|^{2}-\delta_{n}\|x_{n}-x_{n-1}\|^{2}\\
&\leq-(\sigma_{n}-\delta_{n+1})\|x_{n+1}-x_{n}\|^{2}.\\
\end{split}
\end{equation}
Since $0\leq\alpha_{n}\leq\alpha_{n+1}\leq\alpha$,
\begin{equation}\label{Z18}
\begin{split}
\sigma_{n}-\delta_{n+1}
&=\tau(1-\alpha_{n})-\alpha_{n+1}(1+\alpha_{n+1})+\tau(\alpha_{n+1}^{2}-\alpha_{n+1})\\
&\geq\tau(1-\alpha_{n+1})-\alpha_{n+1}(1+\alpha_{n+1})+\tau(\alpha_{n+1}^{2}-\alpha_{n+1})\\
&\geq\tau(1-\alpha)-\alpha(1+\alpha)+\tau(\alpha^{2}-\alpha)\\
&=\tau-2\tau\alpha-\alpha-\alpha^{2}+\tau\alpha^{2}\\
&=-(1-\tau)\alpha^{2}-(1+2\tau)\alpha+\tau.\\
\end{split}
\end{equation}
Combining \eqref{Z17} and \eqref{Z18}, we get
\begin{equation}\label{Z19}
\begin{split}
\Phi_{n+1}-\Phi_{n}\leq-\xi\|x_{n+1}-x_{n}\|^{2},\\
\end{split}
\end{equation}
where $\xi=-(1-\tau)\alpha^{2}-(1+2\tau)\alpha+\tau$. From \eqref{Z09} we know that $\xi>0$.
Therefore,
\begin{equation}\label{Z20}
\begin{split}
\Phi_{n+1}-\Phi_{n}\leq 0.\\
\end{split}
\end{equation}
Thus, the sequence $\{\Phi_{n}\}$ is nonincreasing. Since
\begin{align}
\Phi_{n}&=\|x_{n}-x^*\|^{2}-\alpha_{n}\|x_{n-1}-x^*\|^{2}+\delta_{n}\|x_{n}-x_{n-1}\|^{2}\nonumber\\
&\geq\|x_{n}-x^*\|^{2}-\alpha_{n}\|x_{n-1}-x^*\|^{2},\nonumber\\
\|x_{n}-x^*\|^{2}&\leq\alpha_{n}\|x_{n-1}-x^*\|^{2}+\Phi_{n}\leq\alpha\|x_{n-1}-x^*\|^{2}+\Phi_{1}\nonumber\\
&\leq\cdot\cdot\cdot\leq\alpha^{n}\|x_{0}-x^*\|^{2}+\Phi_{1}(\alpha^{n-1}+\cdot\cdot\cdot+1)\nonumber\\
&\leq\alpha^{n}\|x_{0}-x^*\|^{2}+\frac{\Phi_{1}}{1-\alpha}.\label{Z21}
\end{align}
Similarly, we have
\begin{equation}\label{Z22}
\begin{split}
\Phi_{n+1}&=\|x_{n+1}-x^*\|^{2}-\alpha_{n+1}\|x_{n}-x^*\|^{2}+\delta_{n+1}\|x_{n+1}-x_{n}\|^{2}\\
&\geq-\alpha_{n+1}\|x_{n}-x^*\|^{2}.\\
\end{split}
\end{equation}
Thus, it follows from \eqref{Z21} and \eqref{Z22} that
\begin{align*}
-\Phi_{n+1}\leq\alpha_{n+1}\|x_{n}-x^*\|^{2}\leq\alpha\|x_{n}-x^*\|^{2}\leq\alpha^{n+1}\|x_{0}-x^*\|^{2}+\frac{\alpha\Phi_{1}}{1-\alpha},\\
\end{align*}
and hence from \eqref{Z19} we get
\begin{align*}
\xi\sum_{n=1}^{k}\|x_{n+1}-x_{n}\|^{2}\leq\Phi_{1}-\Phi_{k+1}\leq\alpha^{k+1}\|x_{0}-x^*\|^{2}+\frac{\Phi_{1}}{1-\alpha}\leq\|x_{0}-x^*\|^{2}+\frac{\Phi_{1}}{1-\alpha},\\
\end{align*}
 which implies that $\Sigma_{n=1}^{\infty}\|x_{n+1}-x_{n}\|^{2}<+\infty$. Therefore, $\|x_{n+1}-x_{n}\|\rightarrow0(n\rightarrow\infty)$. Since $\{\alpha_n\}$  is bounded, from \eqref{Z15} we have
$\|x_{n+1}-w_{n}\|\rightarrow 0(n\rightarrow\infty)$. Since
\begin{align*}
0\leq\|w_{n}-x_{n}\|\leq\|x_{n}-x_{n+1}\|+\|x_{n+1}-w_{n}\|,\\
\|w_{n}-x_{n}\|\rightarrow 0\,\,\,\ as\,\,\,\ n\rightarrow \infty.
\end{align*}
In addition,  using Lemma \ref{Lem2.3}, from \eqref{Z14} we have
\begin{align*}
\lim_{n\rightarrow\infty}\|x_{n}-x^*\|=\rho,
\end{align*}
for some $\rho\geq0$. Applying the boundedness of $\{\alpha_n\}$, from \eqref{Z13} we also have
\begin{align*}
\lim_{n\rightarrow\infty}\|w_{n}-x^*\|=\rho.
\end{align*}

\end{proof}
\begin{thy}\label{th4.2} If $A:\mathbb{R}^{n}\rightarrow 2^{\mathbb{R}^{n}}$ is continuous with nonempty compact convex values on $C$ and the suppose $S\neq\emptyset$, then the sequence $\{x_{n}\}$ generated by Algorithm \ref{alg} converges to a solution $\bar{x}$ of \eqref{z01}.
\end{thy}
\begin{proof}Let $x^*\in S$. Since $\mu\in(0,1)$, $(1-\mu)\in(0,1)$. It follows from Lemma \ref{Lem4.1} that
\begin{align*}
0\leq(1-\mu^{2})\|r_{\lambda_{n}}(w_{n},u_{n})\|^{2}\leq\|w_{n}-x^*\|^{2}-\|x_{n+1}-x^*\|^{2}\rightarrow 0\;\;\;as\;\;\;n\rightarrow\infty,
\end{align*}
which implies that
\begin{equation}\label{Z24}
\begin{split}
\lim_{n\rightarrow\infty}\|r_{\lambda_{n}}(w_{n},u_{n})\|^{2}=0.\\
\end{split}
\end{equation}
By the boundedness of $\{x_{n}\}$, there exists a convergent subsequence $\{x_{n_{j}}\}$ converging to $\bar{x}$. By Lemma \ref{Th4.1} (2), there also exists a convergent  subsequence $\{w_{n_{j}}\}$ converging to $\bar{x}$.

If $\bar{x}$ is a solution of the problem \eqref{z01}, i.e., $\bar{x}\in S$. In view of Lemma \ref{Th4.1}, we know that $\lim_{n\rightarrow\infty}\|x_{n}-\bar{x}\|$ exists. Hence, by Lemma \ref{Lem2.4} we have that the sequence $\{x_n\}$ converges to $\bar{x}$.

Suppose now that $\bar{x}$ is not a solution of the problem \eqref{z01}, i.e., $\bar{x}\notin S$. We first show that $m_n$ in Algorithm \ref{alg} cannot tend to $\infty$. Since $A$ is continuous with compact values, Proposition 3.11 in \cite{AE1984}  implies that $\{A(w_{n})|n\in N\}$ is bounded set, and so the sequence $\{u_{n}\}$ is a bounded set. Therefore, there exists a subsequence $\{u_{n_{j}}\}$ converging to $\bar{u}$. Since $A$ is upper semi-continuous with compact values, Proposition 3.7 \cite{AE1984} implies that $A$ is closed, and so $\bar{u}\in A(\bar{x})$. By the definition of $m_{n}$, we have
\begin{align*}
&\gamma^{m_{n}-1}\|u_{n}-\nu\|>\mu\|r_{\gamma^{m_{n}-1}}(w_{n},u_{n})\|.\;\;\;\;\;\;\forall\nu\in A(P_{C}(w_{n}-\gamma^{m_{n}-1}u_{n})).\\
\end{align*}
i.e.,
\begin{align*}
\|u_{n}-\nu\|&>\frac{\mu}{\gamma^{{m_{n}}-1}}\|r_{\gamma^{{m_{n}}-1}}(w_{n},u_{n})\|\\
&\geq\frac{\mu}{\gamma^{{m_{n}}-1}}\min\{1,\gamma^{{m_{n}}-1}\}\|r_{1}(w_{n},u_{n})\|\\
&=\mu\|r_{1}(w_{n},u_{n})\|, \,\,\,\,\,\,\forall\,\nu\in A(P_{C}(w_{n}-\gamma^{{m_{n}}-1}u_{n}))\forall\,\,\,\,m_{n}\geq1,\\
\end{align*}
where the second inequality follows from Proposition \ref{pro02} and the equality follows from $\gamma\in (0,1)$.

If $m_{n_{j}}\rightarrow\infty$ , then $P_{C}(w_{n_{j}}-\gamma^{m_{n_{j}}-1}u_{n})\rightarrow\bar{x}$. By the lower semi-continuity of $A$, we get that there exists $\bar{u}_{n_{j}}\in A(P_{C}(w_{n_{j}}-\gamma^{m_{n_{j}}-1}u_{n_{j}}))$ such that $\bar{u}_{n_{j}}$converges to $\bar{u}$. Therefore,
\begin{equation}\label{Z25}
\begin{split}
\|u_{n_{j}}-\bar{u}_{n_{j}}\|>\mu\|r_{1}(w_{n_{j}},u_{n_{j}})\|\\
\end{split}
\end{equation}
Letting $j\rightarrow\infty$ in \eqref{Z25}, we obtain the contradiction
\begin{align*}
0\geq\mu\|r_{1}(\bar{x},\bar{u})\|>0.
\end{align*}
Therefore, $\{m_{n}\}$ is  bounded, and so is$\{\lambda_{n}\}$. By Proposition \ref{pro02} ,
\begin{equation}\label{Z26}
\begin{split}
\|r_{\lambda_{n}}(w_{n},u_{n})\|\geq\min\{1,\lambda_{n}\}\|r_{1}(w_{n},u_{n})\|=\lambda_{n}\|r_{1}(w_{n},u_{n})\|.
\end{split}
\end{equation}
It follows from \eqref{Z24} and \eqref{Z26} that
\begin{align*}
\lim_{n\rightarrow\infty}\lambda_{n}\|r_{1}(w_{n},u_{n})\|=0.
\end{align*}
Hence,
\begin{align*}
\lim_{n\rightarrow\infty}\|r_{1}(w_{n},u_{n})\|=0.
\end{align*}
Since $r_{1}(\cdot,\cdot)$ is continuous and the sequences $\{w_{n}\}$ and $\{u_{n}\}$ are bounded, there exists
an accumulation point $(\bar{x},\bar{u})$ of $\{(w_{n},u_{n})\}$ such that $r_{1}(\bar{x},\bar{u})=0$. Hence $\bar{x}$ is a solution of the multi-valued variational inequality \eqref{z01}. Similar to the preceding proof, we obtain that $\{x_{n}\}$ converges to $\bar{x}$.

\end{proof}

\section{Numerical experiments}\label{nume}
\indent
\par
\setcounter{equation}{0}
In this section, we present some numerical experiments for the proposed algorithm. The Matlab codes are run on a PC (with Intel(R) Core(TM) i3-4010U CPU  $@$ 1.70GHZ) under MATLAB Version 8.4.0.150421 (R2014b) Service Pack 1. Now, we apply our algorithms to solve the VIP and compare numerical results with other algorithms.

In the following tables, ``Iter." denotes the number of iterations and ``CPU" denotes the CPU time in seconds. The tolerance $\varepsilon $ means when $\|r_{\mu}(x,w)\|\leq \varepsilon $, the procedure stops. \\

\begin{exm}\label{e4.1} Let
\begin{align*}
C: = \{ {x=(x_1,x_2) \in \mathbb{R}_ + ^2:0\leq x_n\leq10},n=1,2\},
\end{align*}
and $A:K\rightarrow 2^{\mathbb{R}^2}$ be defined by
\begin{align*}
A(x) = \{(x_1^{2}+t,x_2^{2}),\,\,\,\forall\,\,\,x=(x_1,x_2)\in \mathbb{R}^2,\,\,\,\ t\in[0,1/5]\}.
\end{align*}
\end{exm}
It is obvious that $A$ satisfies the assumptions in Theorem \ref{th4.2}. We choose $ \mu=0.98 $, $\gamma=0.91 $, $ \alpha=0.03 $ for our Algorithm \ref{alg}; $ \mu=0.35 $, $ \gamma=0.55 $ for Algorithm 2.1 in \cite{FC2014}; $ \mu=0.54 $, $\gamma=0.74 $ for Algorithm 3.1 in \cite{DLY17}. See Figure 1 and Table 1.
\begin{table}[h]
{\begin{footnotesize}\begin{flushleft}
\textbf {Table 1}  Example 4.1 .
\end{flushleft}
\end{footnotesize}}
\centering
\begin{tabular}{llllllllllll}
\hline
&\multicolumn{2}{l}{Algorithm3.1} & & \multicolumn{2}{l}{Algorithm 2.1\cite{FC2014}} & & \multicolumn{3}{l}{Algorithm 3.1 \cite{DLY17}}&  \\
\cline{2-9}
Tolerance $\varepsilon $ & Iter. & CPU &  &Iter.& CPU &  & Iter. & CPU &  \\
\cline{1-11}
          $10^{-1}$       & 11 &0.0780&& 13 &0.2964&& 15 &0.0780 &  \\
          $10^{-2}$       & 13 &0.0936&& 14 &0.3076&& 29 &0.1560 &  \\
          $10^{-3}$       & 14 &0.0936&& 15 &0.3120&& 42 &0.2340 &  \\
          $10^{-4}$       & 15 &0.0936&&  - &  -   && -  &   -   &  \\
\hline
\end{tabular}
\end{table}
\clearpage
\begin{figure}[h]
\centering
\includegraphics[height=8cm, width=10cm]{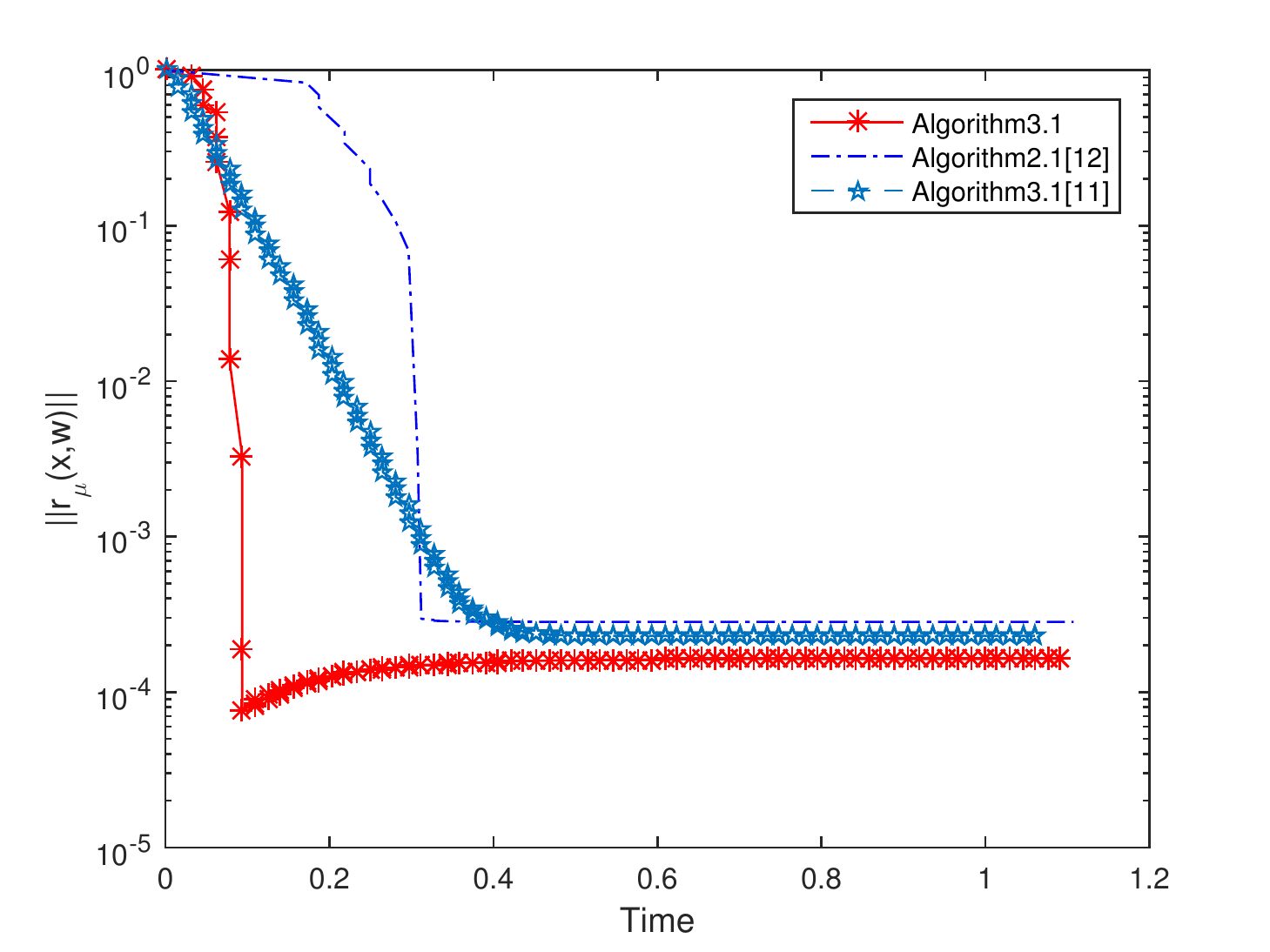}
\caption{$\|x_{n}-x^{*}\|$ and time in Example 4.1}
\end{figure}

\begin{exm}\label{e4.2} Let $n=4$ The feasible set $C$ is given by
\begin{align*}
C: = \left\{ {\left. {x \in {R^n}} \right|\sum\limits_{i = 1}^n {{x_i}=1,-10\leq{x_i}\leq10 ,i = 1, \cdots ,n} } \right\}.
\end{align*}
and $A:K\rightarrow 2^{\mathbb{R}^2}$ be defined by
\begin{align*}
A(x)=\{(t+x_{1},x_{1},x_{1},x_{1}):t\in[\frac{1}{10},\frac{1}{5}]\}
\end{align*}
\end{exm}
Example \ref{e4.2} is tested in \cite{ZFC2017}. It is obvious that $A$ is pseudomonotone and all the assumptions in Theorem \ref{th4.2} are satisfied. We choose $ \mu=0.14 $, $\gamma=0.10 $, $ \alpha=0.72 $ for our Algorithm \ref{alg}; $ \mu=0.52 $, $\gamma =0.49 $ for Algorithm 2.1 in \cite{FC2014}; $ \mu=0.37 $, $\gamma=0.34 $ for Algorithm 3.1 in \cite{DLY17}; See Figure 2 and Table 2.
\begin{table}[h]
{\begin{footnotesize}\begin{flushleft}
\textbf {Table 2}  Example 4.2.
\end{flushleft}
\end{footnotesize}}
\centering
\begin{tabular}{llllllllllll}
\hline
&\multicolumn{2}{l}{Alrithm3.1} & & \multicolumn{2}{l}{Algorithm 2.1\cite{FC2014}} & & \multicolumn{3}{l}{Algorithm 3.1 \cite{DLY17}}&  \\
\cline{2-9}
Tolerance $\varepsilon $ & Iter. & CPU &  &Iter.& CPU &  & Iter. & CPU &  \\
\cline{1-11}
          $10^{-3}$       & 8  &0.0624&& 9  &0.3120&& 28  & 0.1248 &  \\
          $10^{-5}$       & 12 &0.7800&& 18 &0.3900&& 45 & 0.2028 &  \\
          $10^{-7}$       & 16 &0.0936&& 26 &0.4680&& 63 & 0.2808 &  \\
\hline
\end{tabular}
\end{table}
\clearpage
\begin{figure}[h]
\centering
\includegraphics[height=8cm, width=10cm]{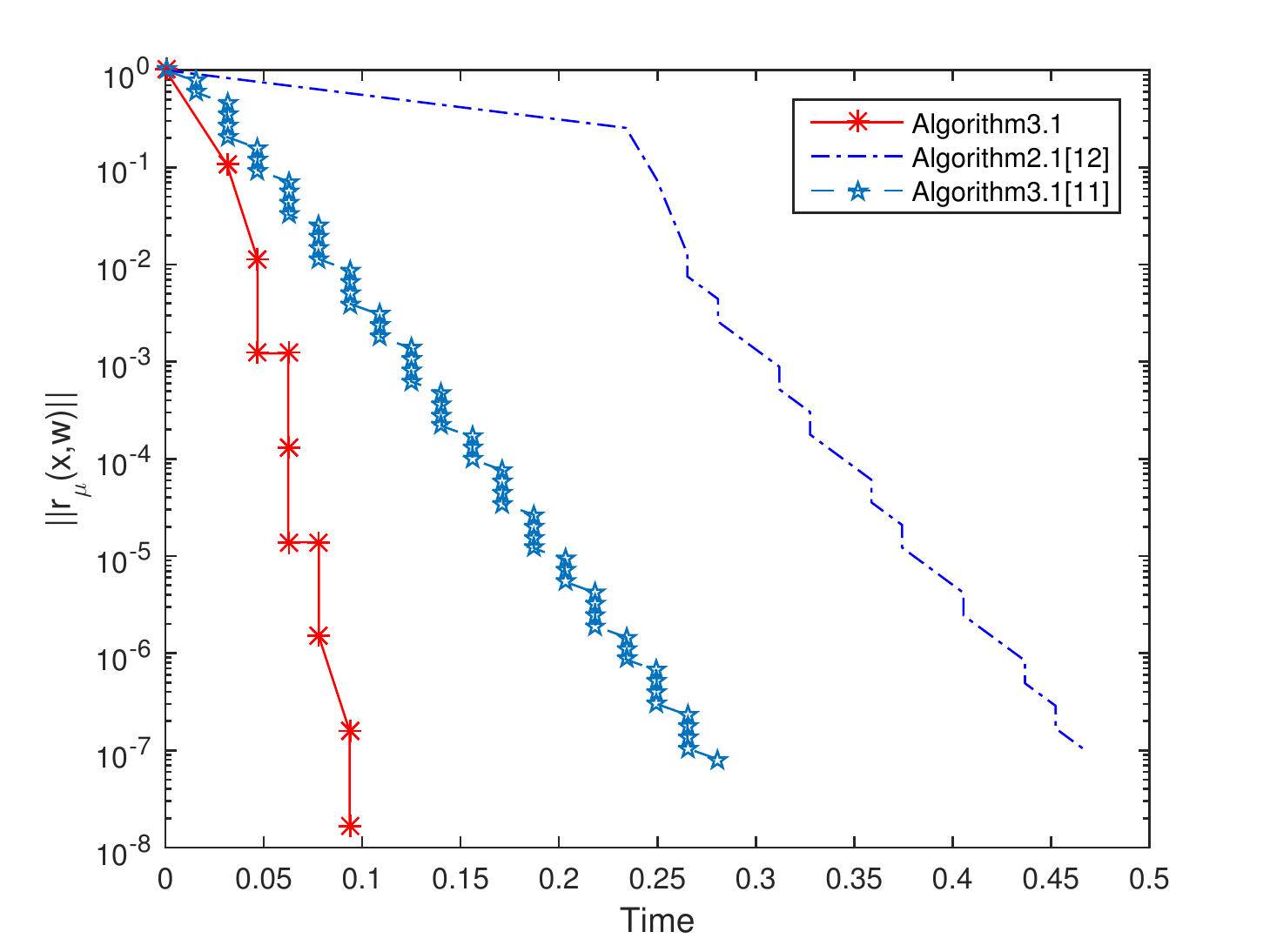}
\caption{$\|x_{n}-x^{*}\|$ and time in Example 4.2}
\end{figure}

\section{Conclusions}
In this paper,  we proposed an inertial  Tseng's  extragradient algorithm for solving multi-valued variational inequalities. We proved the convergence of the
sequences generated by the proposed algorithm and presented some numerical experiments to illustrate the efficiency of our method. Compared with those algorithms in \cite{BM19,HHL19},  only one projection is needed at each iterate in our method. Our method is also different from that in \cite{DLY17}. First, we incorporate the inertial effects in our method. Secondly, the next iterate is related to  Tseng��s technique in our method while in \cite{DLY17} the next iterate is based on contraction method studied in \cite{He97}.

\end{document}